\documentclass[12pt,reqno]{amsart}
 
\usepackage{amsmath,amsthm,amssymb,mathrsfs,amscd,epsfig,color}
\usepackage{url}
\usepackage{graphicx}
\usepackage{mathrsfs}
\usepackage{ulem}
\usepackage{fancybox}
\usepackage{pb-diagram} 
\usepackage[all]{xy}
\usepackage{comment}
\usepackage{mathtools}
\usepackage{centernot}
\usepackage{placeins} 
\usepackage{url}

\makeatletter

\newcommand{\confrac}[2]{%
  \frac{\displaystyle{%
    \strut\hfill{#1}\hfill\;\vrule}}%
      {\displaystyle{%
       \strut\vrule\;\hfill{#2}\hfill}}}%

\makeatletter

    \newcommand\contFrac{\@ifstar{\@contFracStar}{\@contFracNoStar}}

   \def\singleContFrac#1#2{%
        \begin{array}{@{}c@{}}%
            \multicolumn{1}{c|}{#1}%
            \\%
            \hline%
           \multicolumn{1}{|c}{#2}%
        \end{array}%
   }

    \def\@contFracNoStar#1{%
        \mathchoice{
            \@contFracNoStarDisplay@#1//\@nil%
        }{
            \@contFracNoStarInline@#1//\@nil%
        }{
            \@contFracNoStarInline@#1//\@nil%
        }{
            \@contFracNoStarInline@#1//\@nil%
        }%
    }

    \def\@contFracNoStarDisplay@#1//#2\@nil{%
        \@ifmtarg{#2}{%
            #1%
        }{%
            #1+\cfrac{1}{\@contFracNoStarDisplay@#2\@nil}%
        }%
    }

        \def\@contFracNoStarInline@#1//#2\@nil{%
            \@ifmtarg{#2}{%
                #1%
            }{%
                #1 \@@contFracNoStarInline@@#2\@nil%
            }%
        }
        \def\@@contFracNoStarInline@@#1//#2\@nil{%
            \@ifmtarg{#2}{%
                + \singleContFrac{1}{#1}%
            }{%
                + \singleContFrac{1}{#1} \@@contFracNoStarInline@@#2\@nil%
            }%
        }

    \def\@contFracStar#1{%
        \mathchoice{
            \@contFracStarDisplay@#1////\@nil%
        }{
            \@contFracStarInline@#1//\@nil%
        }{
            \@contFracStarInline@#1//\@nil%
        }{
            \@contFracStarInline@#1//\@nil%
        }%
    }

    \def\@contFracStarDisplay@#1//#2//#3\@nil{%
        \@ifmtarg{#2}{%
            #1%
        }{%
            #1 + \cfrac{#2}{\@contFracStarDisplay@#3\@nil}%
        }%
    }

        \def\@contFracStarInline@#1//#2\@nil{%
            \@ifmtarg{#2}{%
                #1%
            }{%
                #1 \@@contFracStarInline@@#2\@nil%
            }%
        }
        \def\@@contFracStarInline@@#1//#2//#3\@nil{%
            \@ifmtarg{#3}{%
                + \singleContFrac{#1}{#2}%
            }{%
                + \singleContFrac{#1}{#2} \@@contFracStarInline@@#3\@nil%
            }%
        }
\makeatother

       \numberwithin{equation}{section}
\theoremstyle{plain}

\newtheorem{thm}{Theorem}[section]
\newtheorem{lem}[thm]{Lemma}

\newtheorem{pro}[thm]{Proposition}
\theoremstyle{definition}

\newtheorem*{prf*}{Proof}

\newtheorem*{pf*}{}
\newtheorem*{lem*}{LemmaA}
\newtheorem*{lm*}{LemmaB}
\makeatletter
\@namedef{subjclassname@2020}{%
  \textup{2020} Mathematics Subject Classification}
\makeatother
\title[A problem of Hirst and sets with restricted slowly growing digits]
{A problem of Hirst for the Hurwitz continued fraction 
and the Hausdorff dimension of sets with restricted slowly growing digits}
\author{Yuto Nakajima and Hiroki Takahasi}

\address{Faculty of Science and Engineering, Doshisha University, Kyoto, 610-0394, JAPAN}
\email{yunakaji@mail.doshisha.ac.jp}
\address{Keio Institute of Pure and Applied Sciences (KiPAS), 
Department of Mathematics,
Keio University, Yokohama,
223-8522, JAPAN} 
\email{hiroki@math.keio.ac.jp}

\begin{document}

\begin{abstract} 
We address the problem of determining the Hausdorff dimension of sets consisting of complex irrationals whose complex continued fraction digits satisfy  
  prescribed restrictions and growth conditions. For the Hurwitz continued fraction, 
  we confirm Hirst's conjecture, 
  as a complex analogue of the result of Wang and Wu [Bull. Lond. Math. Soc. {\bf 40} (2008), no. 1, 18--22] for the regular continued fraction.
We also prove a complex analogue   
of the second-named author's result on the Hausdorff dimension of sets with restricted slowly growing digits [Proc. Amer. Math. Soc. {\bf 151} (2023), no. 9, 3645--3653]. 
To these ends, we 
exploit an infinite conformal iterated function system associated with the Hurwitz continued fraction.
\end{abstract}

\maketitle
\section{Introduction}
Continued fractions of real and complex numbers can provide good rational approximations to irrational numbers, and
various aspects of irrational numbers are understood by means of their continued fraction expansions. 
Conversely, sets of irrationals 
whose continued fraction digits satisfy certain prescribed conditions often become fractal sets.
A principal aim of the dimension theory of continued fractions 
is to determine the Hausdorff dimension of these fractals.
A pioneering result  in this direction
is due to Jarn\'ik \cite{Jarnik1}
for the regular continued fraction  
 \[x=\confrac{1 }{a_{1}(x)} + \confrac{1 }{a_{2}(x)}+ \confrac{1 }{a_{3}(x)}  +\cdots,\]
where $x\in(0,1)$ is an irrational and $a_n(x)\in\mathbb N$ for every $n\geq1$. It is well-known that $x$ is badly approximable if and only if the sequence $(a_n(x))_{n\in\mathbb N}$ is bounded. In \cite{Jarnik1},
Jarn\'ik proved that
the set of badly approximable numbers in $(0,1)$ is of Hausdorff dimension $1$.
On the basis of Jarn\'ik's techniques, Good 
\cite[Theorem~1]{Good1941} proved that the set of irrationals in $(0,1)$ 
whose regular continued fraction digits diverge to infinity is of Hausdorff dimension $1/2$.

As a refinement of Good's theorem, Hirst \cite{Hir73} considered cases where digits
 are restricted to belong to some subset of $\mathbb N$.
 For an infinite subset $B$ of $\mathbb N$,
 he introduced the set
 \[E(B)=\left\{x\in(0,1)\setminus\mathbb Q\colon a_n(x)\in B\text{ for every } n\in\mathbb N\text{ and } \lim_{n\to\infty}a_n(x)=\infty\right\},\]
 and conjectured that for an arbitrary $B$, $E(B)$ is of Hausdorff dimension
  $\tau(B)/2$ where $\tau(B)$ denotes the {\it convergence exponent} given by
  \[\tau(B)=\inf\left\{t\geq0\colon\sum_{k\in B} k^{-t}<\infty\right\}.\] Hirst indeed proved the upper bound for an arbitrary $B$, and proved the lower bound in
 the special case $B=\{k^b\}_{k\in\mathbb N}$, $b\in\mathbb N$
 (see \cite[Theorem~3]{Hir73}).
 Cusick \cite[Theorem~1]{Cus90} proved the lower bound 
 under the following density assumption on $B$ \cite[p.280]{Cus90}: there exist positive constants $c$, $q$ and $r$ such that $r<q\tau(B)$, and for all real numbers $p\geq q$, 
  $\#(B\cap[(n-1)^p,n^p])\geq cn^{p\tau(B)-r}$ for every $n\in\mathbb N$.
   Hirst's conjecture was confirmed
   by
 Wang and Wu \cite[Theorem~1.1]{WW08} without any assumption on $B$, 35 years later than the appearance of Hirst's paper \cite{Hir73}. Hirst's conjecture was then formulated and confirmed in several different setups:
 continued fractions over the field of Laurent series \cite{HuWu09};  infinite iterated function systems on compact intervals \cite{CWW13};  generalized continued fractions \cite{TT16}.





One ramification of Hirst's conjecture is a problem to determine the Hausdorff dimension of subsets of $E(B)$ with prescribed slowly growing speeds of digits, namely sets of the form
 \[G(B,f)=\left\{x\in E(B)\colon a_n(x)\leq f(n)\text{ for every }n\in\mathbb N\right\},\]
 where $f$ is a function taking values in $[\min B,\infty)$ and satisfying $\lim_{n\to\infty}f(n)=\infty$.  
It was proved in \cite[Theorem~1.1]{T23} that $G(B,f)$ is of Hausdorff dimension $\tau(B)/2$ for any such $f$.
This result was generalized to 
 infinite iterated function systems on compact intervals \cite{GHST},
 and to semi-regular continued fractions \cite{NT}.

It is natural to ask for analogous dimension results for complex continued fractions, i.e., continued fractions with digits in complex numbers. Since the work of Adolf Hurwitz \cite{H87} in the 19th century, there have been a number of attempts to define good continued fraction expansions for complex numbers, 
see e.g., \cite{DN,H1902,KST76,Lak73,LeV52,S75}. 
Extending results for the regular continued fraction to complex continued fractions is in general desirable, but not always possible, see e.g., \cite[Section~4]{Lak75}.
The aim of this paper is to confirm Hirst's conjecture for the Hurwitz continued fraction (Theorem~\ref{Main1}), and prove an analogue of the result \cite[Theorem~1.1]{T23} on the dimension of sets with restricted slowly growing digits 
for the Hurwitz continued fraction (Theorem~\ref{Main2}).

\subsection{Statements of the results}
Let $\mathbb Q(\sqrt{-1})$ denote the Gaussian field and let 
$\mathbb Z(\sqrt{-1})$ denote the ring of Gaussian integers, namely
\[\mathbb Q(\sqrt{-1})=\{x+\sqrt{-1}y\in\mathbb C\colon x,y\in\mathbb Q\}\] 
and
\[\mathbb Z(\sqrt{-1})=\{x+\sqrt{-1}y\in\mathbb C\colon x,y\in\mathbb Z\}.\]
For convenience we identify $\mathbb C$ with $\mathbb R^2$ in the obvious way. Then $\mathbb Z(\sqrt{-1})$ is identified  with $\mathbb Z^2$.
Let \[U=\left\{ x+\sqrt{-1}y\in \mathbb C\colon -\frac{1}{2}\le x<\frac{1}{2},\ -\frac{1}{2}\leq y< \frac{1}{2}\right\}.\]
For $x\in\mathbb R$, let $\lfloor x\rfloor$ denote the largest integer not exceeding $x$.
For $x,y\in\mathbb R$ and $z=x+\sqrt{-1}y\in \mathbb C$ we define \[\lfloor z\rfloor=\left\lfloor x+\frac{1}{2}\right\rfloor+\sqrt{-1} \left\lfloor y+\frac{1}{2}\right\rfloor,\]
which is one of the at most four Gaussian integers that are nearest to $z$ in the Euclidean metric.
Define the {\it Hurwitz map} $H\colon U\setminus\{0\}\rightarrow U$ by \begin{equation}\label{H-map}H(z)=\frac{1}{z}-\left\lfloor \frac{1}{z}\right\rfloor.\end{equation}
For each complex number $z\in U\setminus\{0\}$, define a sequence $c_1(z)$, $c_2(z),\ldots$ of nonzero Gaussian integers 
inductively by 
\[c_{n}(z)=\left\lfloor \frac{1}{H^{n-1}(z)}\right\rfloor.\]
If $H^{n}(z)=0$ for some $n\geq1$ then $c_{n+1}(z)$ is not defined.
We have $H^n(z)\neq0$ for every $n\geq1$
if and only if $z\in U\setminus\mathbb Q(\sqrt{-1})$, and in this case $z$
   has the unique infinite expansion of the form 
\begin{equation}\label{H-exp}z=\confrac{1 }{c_{1}(z)} + \confrac{1 }{c_{2}(z)}+ \confrac{1 }{c_{3}(z)}  +\cdots,\end{equation}
see \cite[Theorem~6.1]{DN} for example.
The expansion of $z$ in \eqref{H-exp} is called the
 {\it Hurwitz continued fraction expansion} \cite{H87}.

Substantial dimension results are emerging for the Hurwitz continued fraction.
Generalizing Jarn\'ik's techniques, Gonz\'alez Robert \cite{Ger} proved that the set of complex irrationals whose Hurwitz continued fraction digits diverge to infinity is of Hausdorff dimension $1$. 
Bugeaud et al. \cite{BGH25} obtained a comprehensive dimension result as a complex analogue of the result of Wang and Wu \cite{WW08b} for the regular continued fraction.
Our first result confirms Hirst's conjecture, 
a complex analogue of \cite[Theorem~1.1]{WW08}.
For a subset $S$ of $\mathbb Z(\sqrt{-1})$, 
let $|S|=\{|z|\colon z\in S\}$. 
\begin{thm}
\label{Main1}
For any infinite subset $S$ of $\mathbb Z(\sqrt{-1})$, the set
\[F(S)=\left\{z\in U\setminus\mathbb Q(\sqrt{-1})\colon c_n(z)\in S\text{ for every $n\in\mathbb N$}\text{ and }\lim_{n\to\infty}|c_n(z)|=\infty\right\}\]
is of Hausdorff dimension $\tau(|S|)/2$.
\end{thm}

Computing the convergence exponent of the set $|\mathbb Z(\sqrt{-1})|$ and applying Theorem~\ref{Main1} to the case $S=\mathbb Z(\sqrt{-1})$ yields an alternative proof of  Gonz\'alez Robert's extension \cite{Ger} of Good's theorem \cite[Theorem~1]{Good1941} to the Hurwitz continued fraction.
\begin{thm}[{\cite[Theorem~1.3]{Ger}}]\label{Main-cor} The set
\[\left\{z\in U\setminus\mathbb Q(\sqrt{-1})\colon \lim_{n\to\infty}|c_n(z)|=\infty\right\}\]
is of Hausdorff dimension $1$.
\end{thm}
We now move on to sets with restricted slowly growing digits. 
Our second result is a complex analogue of \cite[Theorem~1.1]{T23}.

\begin{thm}
\label{Main2}
For any infinite subset $S$ of $\mathbb Z(\sqrt{-1})$ and any function $f\colon\mathbb{N}\to [\min |S|, \infty)$ satisfying $\lim_{n\to \infty}f(n)=\infty$,   
the set
\[F(S, f)=\{z\in F(S)\colon |c_n(z)|\leq f(n)\text{ for every }n\in\mathbb N\}\]
is of Hausdorff dimension $\tau(|S|)/2$.
\end{thm}

\subsection{Outline of proofs of the main results}

 The regular continued fraction is generated by iterations of the Gauss map $x\in(0,1]\mapsto 1/x-\lfloor 1/x \rfloor\in[0,1)$. The natural infinite partition $\{(1/(i+1),1/i]\colon i\in\mathbb N\}$ of $(0,1]$ into $1$-cylinders is an infinite Markov partition for the Gauss map.
 For the Hurwitz continued fraction,
 a main difficulty is that the natural infinite partition of the domain $U$ into $1$-cylinders is not a Markov partition for the Hurwitz map $H$: there are finitely many $1$-cylinders whose $H$-images do not cover $U$. 
 As far as Theorems~\ref{Main1} and \ref{Main2} are concerned,
 we may exclude all these problematic $1$-cylinders from consideration.

A key concept underlying proofs of our main results is that of {\it iterated function system} (IFS).
An IFS is a collection of uniformly contracting maps: at each step one of the maps in the collection is applied (see e.g., \cite[Chapter~9]{Fal14}).
An IFS consisting of conformal maps is called a {\it conformal IFS} (see e.g., \cite[Section~6]{MauUrb}).
Level sets of a conformal IFS are almost balls, and hence can be used to compute Hausdorff dimension.
The sets $F(S)$ and $F(S,f)$ in Theorems~\ref{Main1} and \ref{Main2} are Cantor-like sets that may be well described as subsets of the limit set of a two-dimensional conformal IFS.
We provide upper bounds of these two sets using a conformal IFS obtained by removing all the problematic $1$-cylinders.


Providing lower bounds of Hausdorff dimension of $F(S)$ and $F(S,f)$ is much more difficult. 
Arguments in the lower bounds of Hausdorff dimension in the earlier related papers \cite{CWW13,GHST,HuWu09,T23,TT16,WW08} 
rely on the topological nature of intervals, and do not immediately generalize to our two-dimensional setup. 
Our lower bound relies on an ingenious application of the dimension theory of non-autonomous conformal IFSs developed by Rempe-Gillen and Urba\'nski \cite{RU}.
A {\it non-autonomous IFS} is a sequence of collections of contracting maps: unlike the usual IFS, the collection of contractions applied at each step is allowed to vary.
We construct non-autonomous IFSs well-adaped to the prescribed restrictions and growth conditions, estimate the associated pressure functions, and then apply the non-autonomous version of Bowen's formula \cite{Bow79}. 
This approach traces back to our earlier paper \cite{NT} on semi-regular continued fractions \cite{Bos87,DK99,Kra91,Per50}, and may be applicable to diverse setups including other complex continued fractions.

\subsection{Organization of the paper}
The rest of this paper consists of three sections.
In $\S$2 we introduce conformal IFSs, and then state a dimension result on an arbitrary $2$-decaying conformal IFS (Theorem~\ref{Main100}). In $\S$3 we introduce non-autonomous conformal IFSs, and prove Theorem~\ref{Main100}. 
In $\S$4 we construct a $2$-decaying conformal IFS associated with the Hurwitz continued fraction, and apply Theorem~\ref{Main100} to deduce Theorems~\ref{Main1} and \ref{Main2}.
Theorem~\ref{Main-cor} is also proved in $\S$4.


\section{Preliminaries}
This section is a preliminary on conformal IFSs. In $\S$\ref{CIFS} we introduce conformal IFSs, and
 state a dimension result for an arbitrary $2$-decaying conformal IFS. In $\S$\ref{univ-sec} we summarize basic properties of univalent functions and conformal IFSs. 

\subsection{Definition of Conformal IFS}\label{CIFS}
For a holomorphic function $\phi\colon\Omega\to\mathbb R$ and $z\in\Omega$ let $D\phi(z)$ denote the complex derivative of $\phi$ at $z$. 
Throughout this paper, let $\Delta\subset \mathbb C$ be a connected compact set
such that
 the closure of its interior 
coincides with $\Delta$.  
We assume $\Delta$ is a convex set. 
Let $I$ be a countable subset of $\mathbb C$
and let $\{\phi_i\colon i\in I\}$ be a collection of maps from $\Delta$ to itself. For $n\geq2$ and $(i_1,\ldots,i_n)\in I^n$,
write \[\phi_{i_1\cdots i_n}=\phi_{i_1}\circ\cdots\circ\phi_{i_n}.\]
We say $\{\phi_i\colon i\in I\}$ is a {\it conformal IFS} if the following three conditions hold:
\begin{itemize}

\item[(A1)] (open set condition) For any pair $i,j$ of distinct indices in $I$, \[\phi_{i}({\rm int}\Delta)\cap \phi_{j}({\rm int}\Delta)=\emptyset.\]

\item[(A2)] (conformality) There exists a connected open set $\widetilde \Delta\subset\mathbb C$ containing $\Delta$ 
such that each $\phi_i$ extends to a $C^{1}$ conformal diffeomorphism $\widetilde\phi_i\colon\widetilde \Delta\to\widetilde \phi_i(\widetilde \Delta)\subset\widetilde \Delta$.
With a slight abuse of notation, we often write $\phi_i$ for $\widetilde\phi_i$.

\item[(A3)](uniform contraction) There exist  $m\in\mathbb N$ and $\gamma\in (0, 1)$ such that for any $(i_1,\ldots, i_m)\in I^{m}$ and any $z\in \Delta$ we have 
\[0<|D\phi_{i_1\cdots i_m}(z)|\le \gamma.\]
\end{itemize}
If moreover  $\#I=\infty$, then we say $\{\phi_i\colon i\in I\}$ is an {\it infinite conformal IFS} on $\Delta$.

Let $\Phi=\{\phi_i\colon i\in I\}$ be an infinite conformal IFS on $\Delta$. 
Fix $\zeta\in\Delta$. By (A3), we can define an {\it address map} $\Pi\colon I^{\mathbb N}\to \Delta$ by 
\[\Pi((i_n)_{n=1}^{\infty})=\lim_{n\to \infty}\phi_{i_1\cdots i_n}(\zeta).\]
The set $L(\Phi)=\Pi(I^{\mathbb N})$ is called the {\it limit set} of $\Phi$.
Since we do not require
``$\phi_{i}(\Delta)\cap \phi_{j}(\Delta)=\emptyset$
for any pair $i,j$ of distinct indices in $I$'',
the address map may not be injective. 
Let $L'(\Phi)$ denote the set of $z\in L(\Phi)$ such that $\Pi^{-1}(z)$ is a singleton.

 For each $z\in L'(\Phi)$, 
let $(i_n(z))_{n=1}^\infty$ denote the element of the singleton $\Pi^{-1}(z)$. 
We call $(i_n(z))_{n=1}^\infty$ an {\it address} of $z$.
The next condition
 ensures that all but countably many points in the limit set has a unique address 
 (see Lemma~\ref{d-unique}):
\begin{itemize}
\item[(A4)] 
$\bigcup_{i\in I} (\phi_i(\Delta)\cap \partial\Delta)$ is a countable set. 
\end{itemize}

Given an infinite subset $S$ of $I$ and a function $f\colon[\min|S|,\infty)\to\mathbb N$, define
\[F_\Phi(S)=\left\{z\in L'(\Phi)\colon i_n(z)\in S\text{ for every $n\in\mathbb N$}\text{ and }\lim_{n\to\infty}|i_n(z)|=\infty\right\}\] 
and
\[F_\Phi(S,f)=\left\{z\in F_\Phi(S)\colon |i_n(z)|\leq f(n)\text{ for every $n\in\mathbb N$}\right\}.\] 
We aim to compute the Hausdorff dimension of these two sets. This problem is tractable if we assume certain regularity on the decay of derivatives at infinity. We say $\Phi$ is {\it $2$-decaying} if (A4) holds, and there exist positive constants $C_1$, $C_2$ such that
for any $i\in I$ and any $z\in \Delta$, we have 
\begin{equation}\label{CF-der}\frac{C_1}{|i|^{2 }}\le |D\phi_i(z)|\le \frac{C_2 }{|i|^{2} }.\end{equation}
Let $\dim_{\rm H}$ denote the Hausdorff dimension on $\mathbb C$.
A proof of the next theorem is given in $\S$3.

\begin{thm}\label{Main100} 
Let $\Phi=\{\phi_i\colon i\in I\}$ be a 
$2$-decaying conformal IFS on $\Delta$. For any infinite subset $S$ of $I$ and any function $f\colon[\min|S|,\infty)\to\mathbb N$, we have \[\dim_{\rm H}F_\Phi(S)=\dim_{\rm H}F_\Phi(S,f)=\frac{\tau(|S|)}{2}.\] 
\end{thm}
A prime example of a $2$-decaying conformal IFS $\{\phi_i\colon i\in I\}$ is of the form $\phi_i(z)=1/(z+i)$, $i\in I$.
In \cite[Section~6]{MU96}, Mauldin and Urba\'nski considered an infinite conformal IFS of this form on the closed disc centered at the point $1/2\in\mathbb C$ with radius $1/2$, and
proved Bowen's formula for the Hausdorff dimension of the limit set
(see the remark after Theorem~\ref{Bowen}). Condition (A4) holds for this IFS.

One can introduce $\nu$-decaying conformal IFSs $(\nu>1)$ replacing $|i|^2$ in \eqref{CF-der} by $|i|^\nu$, and can generalize Theorem~\ref{Main100}
to $\nu$-decaying conformal IFSs: then
 $\tau(|S|)/2$ should be replaced by $\tau(|S|)/\nu$.
The case $\nu=2$ is most important since it is related to the Hurwitz continued fraction and other various complex continued fractions.

\subsection{Properties of univalent functions and conformal IFSs}\label{univ-sec}
The next lemma is elementary in complex analysis, which follows from Koebe's distortion theorem, see \cite[Theorem~1.4]{CG93} for example. \begin{lem}\label{distortion-lem}Let $\Omega\subset\mathbb C$ be a region and let $A$ be a compact subset of $\Omega$.
  There exists a constant $K\geq1$ such that for every univalent function $\phi\colon\Omega\to\mathbb C$ and every pair of points $z_1$, $z_2$ in $A$ we have
\[\frac{1}{K}\leq\frac{|D\phi(z_1)|}{|D\phi(z_2)|}\leq K.\]\end{lem}
For $z\in\mathbb C$ and $\delta>0$ let
$B_\delta(z)=\{w\in\mathbb C\colon |w-z|< \delta\}$.
The next lemma asserts that the images of open balls under conformal maps with bounded distortion contain open balls of definite diameters related to the derivatives of the maps. Similar statements are well-known and often used implicitly. For explicit presentations, see e.g., \cite[pp.73--74]{MauUrb} and \cite[Lemma~4.2]{N24}.
\begin{lem}\label{conf-lem}
Let $K\geq 1$,
let $\Omega\subset \mathbb C$ be a region and let $\phi\colon \Omega\to \mathbb C$ be univalent. 
For any $z\in \Omega$ and $\delta>0$ such that \[B_{\delta}(z)\subset \Omega \ \text{ and }
\ \sup_{z_1,z_2\in B_{\delta}(z)}\frac{|D\phi(z_1)|}{|D\phi(z_2)|}\leq K,\]
we have
\[B_{\delta|D\phi(z)|/(3K)}(\phi(z))\subset \phi(B_{\delta}(z)).\]\end{lem}
\begin{proof}
It suffices to show 
$B_{\delta|D\phi(z)|/(3K)}(\phi(z))\subset \phi(B_{\delta/2}(z))$.
Since $\phi$ is univalent, it is an open map. If $B_{\delta|D\phi(z)|/(3K)}(\phi(z))\not\subset \phi(B_{\delta/2}(z))$,
then we could take a point $z'\in B_{\delta|D\phi(z)|/(3K)}(\phi(z))\cap(\phi(B_{\delta}(z))\setminus\phi(B_{\delta/2}(z)))\subset\phi(\Omega)$.
Since $\phi^{-1}\colon \phi(\Omega)\to \Omega$ is conformal and its distortion is bounded by $K$ by the hypothesis, we would have 
\[|\phi^{-1}(z')-z|=|\phi^{-1}(z')-\phi^{-1}(\phi(z))|\leq\frac{\delta|D\phi(z)|}{3K}\sup_{w\in \phi(\Omega)}|D\phi^{-1}(w)|\leq \frac{\delta}{3},\]
and so $\phi^{-1}(z')\in B_{\delta/2}(z)$. We would obtain
$z'\in \phi(B_{\delta/2}(z))$, a contradiction.
\end{proof}

For a set $A\subset\mathbb C$ let ${\rm diam}(A)=\sup\{|z_1-z_2|\colon z_1,z_2\in A\}$.
The next lemma summarizes basic properties of conformal IFSs. 
\begin{lem}\label{IFS-lem} Let $\{\phi_i\colon i\in I\}$ be a conformal IFS on $\Delta$.
\begin{itemize}
\item[(a)]There exists $K_0\ge 1$ such that for every $n\in\mathbb N$, every $(i_1,\ldots,i_n)\in I^n$ and any pair of points $z_1, z_2$ in $\widetilde \Delta$, 
\[ \frac{|D\phi_{i_1\cdots i_n}(z_1)|}{|D\phi_{i_1\cdots i_n}(z_2)|} \le K_0.\]

\item[(b)] Let $\zeta\in \Delta$ and $\delta>0$ be such that $B_{\delta}(\zeta)\subset\Delta$. There exist positive constants $K_1$, $K_2$ such that
for every $n\in \mathbb N$ and every $(i_1,\ldots, i_n)\in I^n$
we have
\[
K_1\le \frac{{\rm diam}(\phi_{i_1\cdots i_n}(\Delta))}{|D\phi_{i_1\cdots i_n}(\zeta)|}\leq K_2.\]
\end{itemize}\end{lem}
\begin{proof}

Item (a) follows from Lemma~\ref{distortion-lem} and (A2). From the convexity of $\Delta$, (a) and the conformality of $\phi_{i_1\cdots i_n}$, we have
\[{\rm diam}(\phi_{i_1\cdots i_n}(\Delta ))\leq \max_{z\in \Delta }|D\phi_{i_1\cdots i_n}(z)|\cdot{\rm diam}(\Delta )
\leq K_0|D\phi_{i_1\cdots i_n}(\zeta)|\cdot{\rm diam}(\Delta ).\]
Lemma~\ref{conf-lem} gives
\[B_{\delta|D\phi_{i_1\cdots i_n}(\zeta)|/(3K_0)}(\phi_{i_1\cdots i_n}(\zeta))\subset \phi_{i_1\cdots i_n}(B_{\delta }(\zeta))\subset\phi_{i_1\cdots i_n}(\Delta),\]
and thus
\[{\rm diam}(\phi_{i_1\cdots i_n}(\Delta ))
\geq \frac{2\delta}{3K_0}|D\phi_{i_1\cdots i_n}(\zeta)|.\]
Taking $K_1=2\delta/(3K_0)$ and $K_2=K_0\cdot{\rm diam}(\Delta)$ yields the desired double inequalities in (b).\end{proof}

For an infinite conformal IFS, the address map may not be injective. Moreover, there may be uncountably many limit points with non-unique addresses. The next lemma will be used in $\S$\ref{low-sec} to resolve this issue.
\begin{lem}\label{d-unique}
Let $\Phi=\{\phi_i\colon i\in I\}$ be an infinite conformal IFS on $\Delta$ satisfying (A4). 
Then $L(\Phi)\setminus L'(\Phi)$ is a countable set. 
\end{lem}

\begin{proof}
Let $z\in L(\Phi)$ and suppose $z\notin \bigcup_{n=1}^\infty\bigcup_{\omega\in I^n}\phi_\omega\left(\bigcup_{i\in I} (\phi_i(\Delta)\cap \partial\Delta)\right)$.
If $z\notin L'(\Phi)$, then there exists
 $(j_n)_{n\in\mathbb N}\in I^\mathbb N$
 such that 
 $(i_n)_{n\in\mathbb N}\neq(j_n)_{n\in\mathbb N}$ and
 $z=\Pi((i_n)_{n\in\mathbb N})=\Pi((j_n)_{n\in\mathbb N})$. Put $k=\min\{n\in\mathbb N\colon i_n\neq j_n\}$. If $k>1$ then (A1) implies 
$(\phi_{i_1}\circ\cdots\circ\phi_{i_{k-1}})^{-1}(z ) \in\phi_{i_{k}}(\partial\Delta).$
Applying $\phi_{i_k}^{-1}$ we get
$(\phi_{i_1}\circ\cdots\circ\phi_{i_{k}})^{-1}(z )\in\partial\Delta.$
We also have
$(\phi_{i_1}\circ\cdots\circ\phi_{i_{k}})^{-1}(z )\in\phi_{i_{k+1}}(\Delta).$
Since $\phi_{i_{k+1}}(\Delta)\cap \partial\Delta=\emptyset$ by the hypothesis on $z$, a
 contradiction arises. If $k=1$ then an analogous argument yields a contradiction too. We have verified that
 $L(\Phi)\setminus L'(\Phi)$ is contained in $\bigcup_{n=1}^\infty\bigcup_{\omega\in I^n}\phi_\omega\left(\bigcup_{i\in I} (\phi_i(\Delta)\cap \partial\Delta)\right)$, which is a countable set by (A4).
\end{proof}

\section{Hausdorff dimension of sets for conformal IFSs}
The aim of this section is to prove Theorem~\ref{Main100}. In 
$\S$\ref{conv-sec} and $\S$\ref{low-sec} we establish upper and lower bounds on Hausdorff dimension  for an arbitrary $2$-decaying conformal IFS.
 A proof of the lower bound relies on the dimension theory of non-autonomous conformal IFSs summarized in $\S$\ref{ncifs}. In $\S$\ref{pf-thmc} 
 we complete the proof of Theorem~\ref{Main100}.

\subsection{The upper bound of Hausdorff dimension}\label{conv-sec}
The following upper bound is rather straightforward, from the properties of conformal IFSs and the definition of convergence exponent. 
\begin{pro}\label{upper} 
Let $\Phi=\{\phi_i\colon i\in I\}$ be a 
$2$-decaying conformal IFS on $\Delta$. For any infinite subset $S$ of $I$ we have \[\dim_{\rm H}F_\Phi(S)\leq\frac{\tau(|S|)}{2}.\]\end{pro}

\begin{proof}
Let $\zeta\in \Delta$ and $\delta>0$ be such that $B_{\delta}(\zeta)\subset\Delta$.
Let $\varepsilon>0$. 
Let $N\geq1$
be sufficiently large such that
\begin{equation}\label{upper-eq1}\left(\frac{K_0K_2C_2}{K_1}\right)^{(\tau(|S|)+\varepsilon)/2 }\sum_{\substack{ i\in S\\ |i|\geq N} }|i|^{-\tau(|S|)-\varepsilon}\leq1,\end{equation}
where $K_0$, $K_1=K_1(\delta)$, $K_2$ are the constants in Lemma~\ref{IFS-lem} and $C_2$ is the constant in \eqref{CF-der}.
From collections of level sets
$\phi_{i_1\cdots i_n}(\Delta)$ satisfying 
$i_k\in S$ and $|i_k|\geq N$ for every $1\leq k\leq n$, 
we choose coverings of the set
\[F_{\Phi,N}(S)=\{z\in F_\Phi(S)\colon |i_{n}(z)|\geq N\text{ for every }n\geq1 \},\] and use them to estimate the Hausdorff dimension of $F_{\Phi,N}(S)$ from above. 

 By (a), (b) in Lemma~\ref{IFS-lem} and \eqref{CF-der}, for every $n\in\mathbb N$ and every $(i_1,\ldots,i_{n+1})\in I^{n+1}$ we have
\[\begin{split}\frac{{\rm diam}(\phi_{i_1\ldots i_ni_{n+1}}(\Delta))}{{\rm diam}(\phi_{i_1,\ldots,i_n}(\Delta))}&\leq
\frac{K_2}{K_1}\frac{|D\phi_{i_1\cdots i_ni_{n+1}}(\zeta)|}{|D\phi_{i_1\cdots i_n}(\zeta)|}\\
&\leq \frac{K_0K_2}{K_1}|D\phi_{i_{n+1}}(\zeta)|\leq\frac{K_0K_2C_2}{K_1}|i_{n+1}|^{-2}.\end{split}\]
Then we have
\[\sum_{\substack{i_{n+1}\in S\\
|i_{n+1}|\geq N }}\frac{{\rm diam}(\phi_{i_1\cdots i_ni_{n+1}}(\Delta))^{ (\tau(|S|)+\varepsilon)/2  } }
{{\rm diam}(\phi_{i_1\cdots i_n}(\Delta))^{(\tau(|S|)+\varepsilon)/2}}\leq \left(\frac{K_0K_2C_2}{K_1}\right)^{(\tau(|S|)+\varepsilon)/2
     }\!\!\!\!\!\!\!\sum_{\substack{i_{n+1}\in S\\
|i_{n+1}|\geq N }}|i_{n+1}|^{-\tau(|S|)-\varepsilon},\]
which does not exceed $1$ by \eqref{upper-eq1}.
It follows that 
\[\frac{\sum_{\substack{i_1,\ldots,i_{n+1}\in S\\ |i_1|\geq N,\ldots,|i_{n+1}|\geq N } }{\rm diam}(\phi_{i_1\cdots i_{n+1}}(\Delta))^{(\tau(|S|)+\varepsilon)/2  } }{ \sum_{\substack{i_1,\ldots,i_{n}\in S\\ |i_1|\geq N,\ldots,|i_{n}|\geq N } }{\rm diam}(\phi_{i_1\cdots i_{n}}(\Delta))^{(\tau(|S|)+\varepsilon)/2}}\leq1,\]
and therefore
\[\sum_{\substack{i_1,\ldots,i_{n+1}\in S\\ |i_1|\geq N,\ldots,|i_{n+1}|\geq N } }
{\rm diam}(\phi_{i_1\cdots i_{n+1}}(\Delta))^{(\tau(|S|)+\varepsilon)/2 }\leq1.\]
Since $\sup\{{\rm diam}(\phi_{i_1\cdots i_n}(\Delta))\colon (i_1,\ldots, i_n)\in I^n\}\to0$ as $n\to\infty$ by Lemma~\ref{IFS-lem}(b) and (A3), we obtain
$\dim_{\rm H} F_{\Phi,N}(S)\leq(\tau(|S|)+\varepsilon)/2.$

Note that
$F_\Phi(S)\subset F_{\Phi,N}(S)\cup\bigcup_{n=1}^\infty\bigcup_{(i_1,\ldots, i_n)\in I^n} \phi_{i_1\cdots i_n}(F_{\Phi,N}(S))$.
Since each $\phi_i$ is Lipschitz continuous and Hausdorff dimensions do not change under the action of bi-Lipschitz homeomorphisms, it follows that
$\dim_{\rm H} F_\Phi(S)\leq \dim_{\rm H} F_{\Phi,N}(S)\leq(\tau(|S|)+\varepsilon)/2$. Since $\varepsilon>0$ is arbitrary, we obtain $\dim_{\rm H} F_\Phi(S)\leq\tau(|S|)/2$ as required.
\end{proof}

\subsection{Dimension theory for non-autonomous conformal IFSs}\label{ncifs}
For the sake of the lower bound in the proof of Theorem~\ref{Main100}, we summarize 
the dimension theory \cite{RU} for non-autonomous conformal IFS on the Euclidean space $\mathbb R^d$ of dimension $d\geq1$.
We note that $d=2$ throughout our application of this theory in $\S$\ref{low-sec}.

Let $W\subset \mathbb{R}^d$ be an open set and let $\phi\colon W\rightarrow \phi(W)$ be a diffeomorphism. We say $\phi$ is {\it conformal} if for any $x\in W$ the differential $D\phi(x)\colon\mathbb{R}^d\rightarrow \mathbb{R}^d$ is a similarity linear map: 
$D\phi(x)=c_x\cdot M_x$ where $c_x>0$ is a scaling factor at $x$ and $M_x$ is a $d\times d$ orthogonal matrix. For a conformal map $\phi\colon W\to\phi(W)$
and a set $A\subset W$, we set
\[\Vert D\phi\Vert_A=\sup\{|D\phi(x)|\colon x\in A\},\]
where $|D\phi(x)|$ denotes the scaling factor of $\phi$ at $x$. 



For each $n\in \mathbb N$ let $I^{(n)}$
be a finite set. We introduce index sets
\begin{equation}\label{index-set}I^{\infty}=\prod_{j=1}^{\infty}I^{(j)}, \text{ and } I^q_n=\prod_{j=n}^q I^{(j)} \text{ for an integer }q\geq n. 
\end{equation}
For each $n\in\mathbb N$
let $\{\phi_i^{(n)}\colon i\in I^{(n)}\}$ be a finite collection 
of self maps of $\Delta$. 
For $\omega=i_1i_2\cdots\in I^\infty$ and
 $n, q\in \mathbb{N}$ with $n\leq q$, we set \begin{equation}\label{indexII}\omega|_n^{q}=i_n\cdots i_{q}\in I_n^{q}\ \text{ and }\
\phi_{\omega|_n^q}=\phi_{i_n}^{(n)}\circ\cdots\circ \phi_{i_{q}}^{(q)}.\end{equation}

 A {\it non-autonomous conformal IFS} on $\Delta$ is a sequence 
$\Phi=(\Phi^{(n)})_{n=1}^{\infty}$,
$\Phi^{(n)}=\{\phi_i^{(n)}\colon i\in I^{(n)}\}$
of collections of self maps of $\Delta$ 
which satisfy the following four conditions:
\begin{itemize}
\item[(B1)] (open set condition) For every $n\in\mathbb N$ and any pair $i,j$ of distinct indices in $I^{(n)}$,  
\[
\phi_i^{(n)}({\rm int}\Delta)\cap \phi_j^{(n)}({\rm int}\Delta)=\emptyset.
\]

\item[(B2)] (conformality) There exists a connected open set $\widetilde \Delta$ of $\mathbb R^d$ containing $\Delta$ 
such that each $\phi_i^{(n)}$ extends to a $C^{1}$ conformal diffeomorphism $\widetilde{\phi}_i^{(n)}\colon\widetilde \Delta\to\tilde{\phi}_i^{(n)}(\widetilde \Delta)\subset\widetilde \Delta$.
With a slight abuse of notation
we write $\phi_i^{(n)}$ for $\widetilde\phi_i^{(n)}$.
\item[(B3)] (bounded distortion) 
There exists $K\ge 1$ such that for all $\omega\in I^\infty$ and all $n$, $q\in \mathbb{N}$ with $n\leq q$, 
\[
|D\phi_{\omega|_n^q}(x_1)|\le K|D\phi_{\omega|_n^q}(x_2)|\ \text{ for all }x_1,x_2\in \widetilde \Delta.
\]

\item[(B4)] (uniform contraction) 
There are constants $0< \gamma <1$ and $L\geq1$ such that for all $\omega\in I^{\infty}$ and all $n,q\in\mathbb N$ with
$q-n\geq L$, 
\[
\|D\phi_{\omega|_n^q}\|_{X}\le\gamma^{q-n+1}.
\]
\end{itemize}


Let $\Phi=(\Phi^{(n)})_{n=1}^{\infty}$ be a non-autonomous conformal IFS on $\Delta$.
Condition (B4) ensures that the set $\bigcap_{n=1}^{\infty}\phi_{\omega|_{1}^n}(\Delta)$ is a singleton for each $\omega\in I^{\infty}$.
We define an
 {\it address map} $\Pi \colon I^{\infty} \to \Delta$ by
\[\Pi(\omega)\in \bigcap_{n=1}^{\infty}\phi_{\omega|_{1}^n}(\Delta),\]
and the {\it limit set} by 
 \[
\Lambda(\Phi)=\Pi(I^{\infty}).
\]
For $s\geq0$ 
we introduce a partition function
\[
Z_n^{\Phi}(s)=\sum_{\omega\in I_{1}^n}(\Vert D\phi_{\omega}\Vert_{\Delta})^s,
\]
and
a lower pressure function
$\underline{P}^\Phi\colon[0,\infty)\to [-\infty,\infty]$ of $\Phi$ by
\[\underline{P}^{\Phi}(s)=\liminf\limits_{n\rightarrow \infty}\frac{1}{n}\log Z_n^{\Phi}(s).\]
The lower pressure function has the following monotonicity
\cite[Lemma~2.6]{RU}:
if $0\leq s_1<s_2$ then $\underline{P}^\Phi(s_1)=\underline{P}^\Phi(s_2)=\infty$ or
$\underline{P}^\Phi(s_1)=\underline{P}^\Phi(s_2)=-\infty$ or $\underline{P}^\Phi(s_1)>\underline{P}^\Phi(s_2)$. So, one can  define a critical value
\[s(\Phi)={\rm sup}\{s\ge0\colon  \underline{P}^\Phi(s)>0\}={\rm inf}\{s\ge0\colon \underline{P}^\Phi(s)<0\},
\]
called the {\it Bowen dimension}. 
We say the non-autonomous conformal IFS $\Phi$ is {\it subexponentially bounded} if
\[
\lim_{n \to \infty}\frac{1}{n}\log \# I^{(n)}=0.
\]

\begin{thm}[{\cite[Theorem~1.1]{RU}}]
\label{Bowen}
Let $\Phi$ be a non-autonomous conformal IFS that is subexponentially bounded.
Then  
\[\dim_{\rm H}\Lambda(\Phi)=s(\Phi).
\]
\end{thm}
This type of formula, first established  in \cite{Bow79} for Fuchsian groups without parabolic elements, is called Bowen's formula. Bowen's formula is well-known for conformal IFSs \cite{MU96,MauUrb}. Theorem~\ref{Bowen} is an extension to non-autonomous conformal IFSs.


\subsection{The lower bound of Hausdorff dimension}\label{low-sec}
Using the dimension theory of non-autonomous conformal IFSs summarized in $\S$\ref{ncifs}, we prove
the following lower bound on Hausdorff dimension.
\begin{pro}
\label{lower}
Let $\Phi=\{\phi_i\colon i\in I\}$ be a 
$2$-decaying conformal IFS on $\Delta$.  Let $S$ be an infinite subset of $I$ and let $f\colon\mathbb{N}\rightarrow [\min |S|, \infty)$ satisfy $\lim_{n\to \infty}f(n)=\infty.$ 
Then \[\dim_{\rm H}F_\Phi(S, f)\ge \frac{\tau(|S|)}{2}.\]
\end{pro}

\begin{proof}An idea is to
 extract a family of non-autonomous conformal IFSs on $\Delta$ from $\Phi$
 whose limit sets are of Hausdorff dimension approximately equal to $\tau(|S|)/2$. This idea traces back to \cite{NT}.
We may assume $\tau(|S|)>0$, for otherwise there is nothing to prove.
Let $\varepsilon\in (0, \tau(|S|))$. 

We choose a sequence $(z_m)_{m\in \mathbb{N}}$ in $S$ inductively as follows:
 \begin{equation}\label{cccc}|z_1|=\min\{|i|\colon i\in S\}.\end{equation} $|z_{m+1}| > |z_m|$
for every $m\geq1$ and
\begin{equation}
\label{aaaa}
\sum_{\substack{i\in S\\|z_m|\le |i |< |z_{m+1}|}}|i |^{-\tau(|S|)+\varepsilon}\ge 1. 
\end{equation}
Set 
\begin{equation}
\label{defbm}
 S_m=
\begin{cases}
 \{i\in S\colon|i|=|z_1|\}&\text{for } m=1,\\

\{i\in S\colon|z_m|\le |i|< |z_{m+1}|\}&\text{for }m\ge 2.
 \end{cases}
 \end{equation}
Let $(t_m)_{m\in \mathbb{N}}$ be a sequence of positive integers such that for every $m\ge 2$ we have
\begin{equation}
\label{hutou}
|z_{m+1}|\le \inf\left\{f(n)\colon \sum_{j=1}^{m-1}t_j+1\le n \le  \sum_{j=1}^{m}t_j\right\},
\end{equation}
and
\begin{equation}
\label{hutou2}
\lim_{m\to \infty}\frac{\log \# S_m}{ \sum_{j=1}^{m}t_j}=\lim_{m\to \infty}\frac{\log \# S_m}{ \sum_{j=1}^{m-1}t_j+1}=0.
\end{equation}
Since $\lim_{n\to \infty} f(n) = \infty$, one can choose such a sequence by induction on $m$. 

For each integer $1\leq n\leq t_1$, we set
\[I^{(n)}=S_1.\]
For each 
integer $n\geq t_1+1$ we pick $m\ge 2$ such that $\sum_{j=1}^{m-1}t_j+1\le n \le  \sum_{j=1}^{m}t_j$,
and set \[I^{(n)}=S_m.\] Put 
$I^{\infty}=\prod_{j=1}^{\infty}I^{(j)}$ and
$I^k_n=\prod_{j=n}^k I^{(j)}$
as in \eqref{index-set}. 
For $n\geq t_1+1$ we have \[I_1^n=(S_1)^{t_1}\times \cdots \times (S_{m-1})^{t_{m-1}}\times 
(S_m)^{l_n},\]
where $l_n=n-\sum_{j=1}^{m-1}t_j$ and  
\[I^{\infty}=(S_1)^{t_1}\times (S_2)^{t_2}\times \cdots \times (S_m)^{t_m}\times \cdots.\] 
Let $n\in\mathbb N$. For each $i\in I^{(n)}$ we put
$\phi_{i}^{(n)}=\phi_{i}$,
and set
\[\Psi^{(n)}=\{\phi^{(n)}_{i}\colon i\in I^{(n)}\}.\] 
For $\omega=i_1i_2\cdots\in I^\infty$ and
 $n, q\in \mathbb{N}$ with $n\leq q$, we set $\omega|_n^{q}=i_n\cdots i_{q}\in I_n^{q}$ and 
$\phi_{\omega|_n^q}=\phi_{i_n}^{(n)}\circ\cdots\circ \phi_{i_{q}}^{(q)}$
as in \eqref{indexII}.
Finally we set
\[\Psi=(\Psi^{(n)})_{n=1}^\infty.\]
Clearly $\Psi$ is a non-autonomous conformal IFS on $\Delta$. 
\begin{lem}\label{NCIFS-lem} The non-autonomous conformal IFS $\Psi$ is subexponentially bounded, and $\Lambda(\Psi)$ 
is contained in $F_\Phi(S, f)$.
\end{lem}

\begin{proof} 







Recall that for each integer $n\geq t_1+1$ we have 
 $\sum_{j=1}^{m-1}t_j+1\le n $.
Then
\[0\leq \frac{1}{n}\log \# I^{(n)}\le \frac{\log \# S_m}{\sum_{j=1}^{m-1}t_j+1}.\]
From this and \eqref{hutou2} we obtain
$\lim_{n\to \infty}(1/n)\log\# I^{(n)}=0$, namely $\Psi$ is subexponentially bounded.


Let $z\in\Lambda(\Psi)$.
Since
there exists $(i_n)_{n\in\mathbb N}\in I^\mathbb N$ such that 
$i_n\in S$ for every $n\geq1$ and $\Pi((i_n)_{n\in\mathbb N})=z$,
 Lemma~\ref{d-unique} implies $z\in L'(\Phi)$ except for countable number of points.
The first alternative of \eqref{defbm} gives \[|i_{n}(z)|=|z_1|\le f(n)\ \text{ for every }1\leq n\leq t_1.\] For every $m\ge 2$, the second alternative of \eqref{defbm} and \eqref{hutou} yield 
\[|z_m|\le |i_{n}(z)|<|z_{m+1}|\le f(n)\ \text{ for every } \sum_{j=1}^{m-1}t_j+1\leq n\leq \sum_{j=1}^{m}t_j.\] 
As $n\to\infty$ we have $m\to\infty$, $|z_m|\to\infty$ and so $|i_{n}(z)|\to \infty$.
Hence we obtain
$z\in F_\Phi(S, f)
$.
\end{proof}

Recall that the lower pressure function $\underline{P}^\Psi\colon[0,\infty)\to[-\infty,\infty]$ 
  is
 given by \[\underline{P}^\Psi(s)=\liminf\limits_{n\rightarrow \infty}\frac{1}{n}\log Z_n^\Psi(s),\text{ where }
 Z_n^\Psi(s)=\sum_{\omega\in I_1^n}(\Vert D\phi_{\omega}\Vert_{\Delta })^s.\]
 By Lemma~\ref{NCIFS-lem}
and Theorem~\ref{Bowen},
the Bowen dimension $s(\Psi)$ satisfies 
\begin{equation}\label{eqp-1}\dim_{\rm H}F_\Phi(S,f)\geq\dim_{\rm H}\Lambda(\Psi)=
s(\Psi).\end{equation}
In order to estimate the Bowen dimension from below, we estimate the lower pressure function from below.
By \eqref{CF-der}, for every $n\geq1$ and every $i\in I^{(n)}$ we have
\begin{equation}
\label{control}
\min_{z\in\Delta}|D\phi_{i}^{(n)}(z)| \ge \frac{C_1}{|i|^2}.
\end{equation}
Since $|z_m|\to\infty$ as $m\to\infty$, for any $\delta>0$ there exists $N\geq1$ such that for all $i\in I$ with $|i|\geq |z_{N+1}|$ we have
\begin{equation}\label{control2}
\frac{C_1 }{|i|^2}\geq \frac{1}{|i|^{2+\delta}}.\end{equation}
Using the chain rule,  \eqref{control} and \eqref{control2} 
we have
\[\begin{split}
Z_n^\Psi(s)=&
\sum_{(i_1,\ldots,i_n)\in I_1^n}(\Vert D(\phi_{i_1}^{(1)}\circ \phi_{i_2}^{(2)}\circ\cdots \circ \phi_{i_n}^{(n)})\Vert_{\Delta })^s\\\ge& 
C_1^{Ns}\sum_{(i_1,\ldots,i_n)\in I_1^n}{|i_1|}^{-2s}\cdots
{|i_{t_1+\cdots +t_N}|}^{-2s}{|i_{t_1+\cdots +t_N+1}|}^{-(2+\delta)s}\cdots{|i_n|}^{-(2+\delta)s}
\\=& 
C_1^{Ns}|z_1|^{-2st_1}\left(\sum_{i\in S_2}|i|^{-2s}\right)^{t_2}\cdots \left(\sum_{i\in S_{N}}|i|^{-2s}\right)^{t_{N}}\\
&\times\left(\sum_{i\in S_{N+1}}|i|^{-(2+\delta)s}\right)^{t_{N+1}}\cdots\left(\sum_{i\in S_m}|i|^{-(2+\delta)s}\right)^{l_n}
.
\end{split}\]
Set
$s_{\varepsilon,\delta}=(\tau(|S|)-\varepsilon)/(2+\delta)$. By \eqref{aaaa}
we have \[\left(\sum_{i\in S_{N+1}}|i|^{-(2+\delta)s_{\varepsilon,\delta}}\right)^{t_{N+1}}\cdots \left(\sum_{i\in S_m}|i|^{-(2+\delta)s_{\varepsilon,\delta}}\right)^{l_n}\geq1.\]
Substituting $s=s_{\varepsilon,\delta}$ and combining the above two inequalities yield
\[Z_n^\Psi(s_{\varepsilon,\delta})\ge C_1^{Ns_{\varepsilon,\delta}}|z_1|^{-2s_{\varepsilon,\delta}t_1}\left(\sum_{i\in S_2}|i|^{-2s_{\varepsilon,\delta}}\right)^{t_2}\cdots \left(\sum_{i\in S_{N}}|i|^{-2s_{\varepsilon,\delta}}\right)^{t_{N}}.\]
Since the right-hand side is independent of $n$, it follows that $\underline{P}^\Psi(s_{\varepsilon,\delta})
\geq0$, and hence 
\begin{equation}\label{eqp-2}
s(\Psi)\ge \frac{\tau(|S|)-\varepsilon}{2+\delta}.
\end{equation}
Combining \eqref{eqp-1} and \eqref{eqp-2}, and 
 letting $\delta\to0$ and then $\varepsilon\to0$ yields
the desired inequality 
in Proposition~\ref{lower}.
\end{proof}

\subsection{Proof of Theorem~\ref{Main100}}\label{pf-thmc}
Let $\Phi=\{\phi_i\colon i\in I\}$ be a $2$-decaying 
conformal IFS on $\Delta$ and let $S$ be an infinite subset of $I$.
Proposition~\ref{upper} gives
$\dim_{\rm H}F_\Phi(S)\leq\tau(|S|)/2$.
Let $f\colon\mathbb{N}\rightarrow [\min |S|, \infty)$ satisfy $\lim_{n\to \infty}f(n)=\infty.$ 
Proposition~\ref{lower} gives $\dim_{\rm H}F_\Phi(S, f)\ge \tau(|S|)/2$.
Since $F_\Phi(S,f)\subset F_\Phi(S)$, the desired equalities hold. \qed
\section{Proofs of the main results}
 
Returning to the Hurwitz continued fraction,
in $\S$\ref{basic-def} we construct a $2$-decaying conformal IFS associated with it.
In $\S$\ref{mainpf} we apply Theorem~\ref{Main100} to this IFS and 
  complete the proofs of Theorems~\ref{Main1} and \ref{Main2}. 
 In $\S$\ref{pfthmcor} we prove Theorem~\ref{Main-cor}.

\subsection{Conformal IFS associated with the Hurwitz continued fraction}\label{basic-def}
Let
$\mathbb D_{1}=\{(k,\ell)\in\mathbb Z^2\colon k^2+\ell^2\geq 2\}$ 
and
$\mathbb D_{2}=\{(k,\ell)\in\mathbb Z^2\colon k^2+\ell^2\geq 8\}.$
For each $i=1,2$ and $n\in\mathbb N$ 
let $\mathbb D_i^n$ denote the set of $n$-strings of elements of $\mathbb D_i$.

For each $(k,\ell)\in\mathbb D_1$ 
define a domain
 \[U_{k,\ell}=\left\{\frac{1}{x+\sqrt{-1}y}\in\mathbb C\colon -\frac{1}{2}\leq x-k<\frac{1}{2},\ -\frac{1}{2}\leq y-\ell<\frac{1}{2} \right\},\]
 which intersects $U$ and
 is bordered by the four circles $(x-1/(2k\pm1))^2+y^2=(1/(2k\pm 1))^2$, $x^2+(y+ 1/(2\ell\pm1))^2=(1/(2\ell\pm1))^2$
 (double sign corresponds) through the origin as indicated in \textsc{Figure}~\ref{fig1}.
 Clearly we have
 \begin{itemize}
 \item[(H1)] if $(k_1,\ell_1)\neq (k_2,\ell_2)$ then $U_{k_1,\ell_1}\cap U_{k_2,\ell_2}=\emptyset$.
\end{itemize}

For $n\in\mathbb N$,
$(c_1,\ldots, c_n)\in  \mathbb D_1^n$ define  an {\it $n$-th cylinder}
\[[c_1,\ldots, c_n]=\{z\in U\colon\ c_i(z)\text{ is well-defined and }c_i(z)=c_i\text{ for every }1\leq i\leq n\}.\]
This
is the set of elements of $U$ which have the finite or infinite Hurwitz continued fraction expansion \eqref{H-exp} beginning with $c_1,\ldots,c_n$. It is not difficult to see the following properties:

\begin{itemize}  
\item[(H2)] each $1$-cylinder 
$[c_1]$, $c_1=(k,\ell)\in\mathbb D_1$
has the form
  \[[c_1]=\begin{cases}
  U_{k,\ell}\cap U\subsetneq U_{k,\ell}&\text{ if $(k,\ell)\in\mathbb D_1\setminus \mathbb D_2$;}\\
  U_{k,\ell}&\text{ if $(k,\ell)\in\mathbb D_2$;}
  \end{cases}\]
\item[(H3)]  
$\bigcup\{[c_1]\colon c_1\in\mathbb D_1\}=U\setminus\{0\}$, 
see \textsc{Figure}~\ref{fig2};
\item[(H4)] if $n\geq2$ then for every
$(c_1,\ldots,c_n)\in\mathbb D_1^n$
we have $H([c_1,\ldots,c_n])\subset [c_2,\ldots,c_n]$, and the equality holds if $(c_1,\ldots,c_n)\in\mathbb D_2^n$.
\end{itemize}

Direct calculations show that $\mathbb D_1\setminus\mathbb D_2$ consists of the following sixteen elements:
$(-1,1)$, $(-1,-2)$, $(0,-2)$, $(1,-2)$, $(1,-1)$, $(2,-1)$, $(2,0)$, $(2,1)$, $(1,1)$, $(1,2)$, $(0,2)$, $(-1,2)$, $(-1,1)$, $(-2,1)$, $(-2,0)$,
$(-2,-1)$, see \textsc{Figure}~\ref{fig2}.
The union of all the corresponding $1$-cylinders contains a neighborhood of $\partial U$. This and (H3) together imply 
 \begin{itemize}\item[(H5)] $\overline{U_{k,\ell}}\subset U$ if and only if $(k,\ell)\in\mathbb D_2$. 
\end{itemize}

\begin{figure}
\begin{center}
\includegraphics[height=7.5cm,width=8cm]{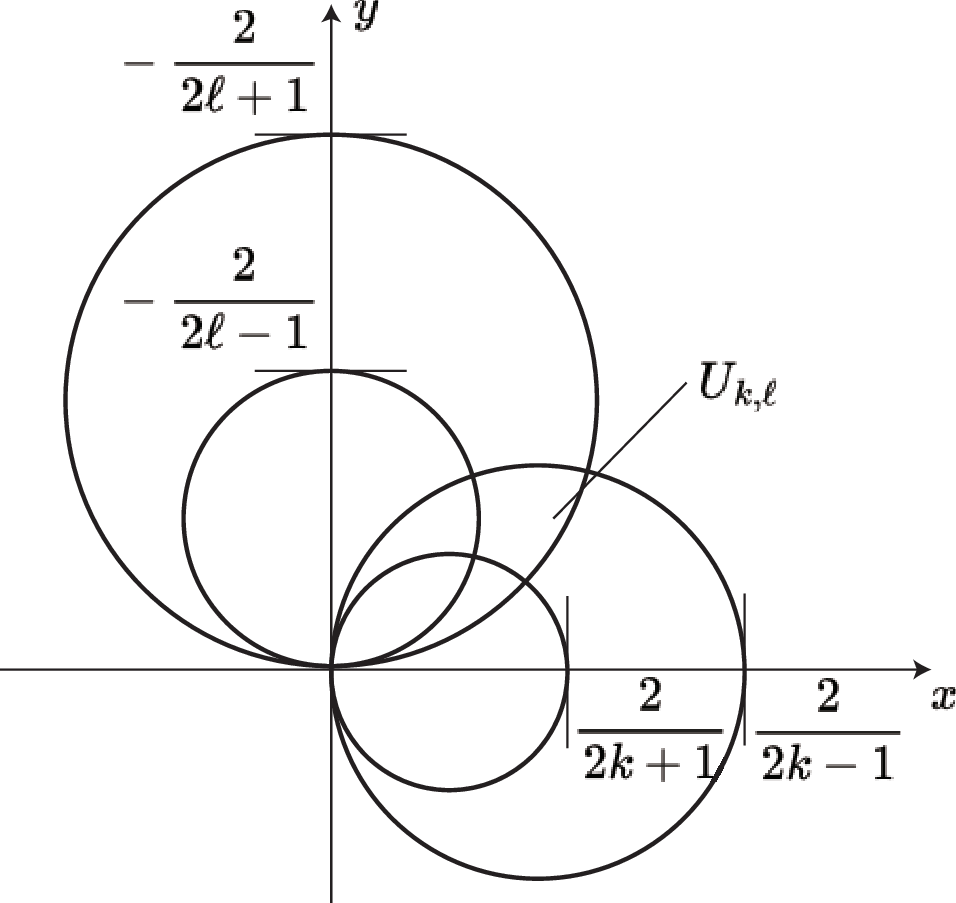}
\caption
{The domain $U_{k,\ell}$ ($(k,\ell)\in\mathbb D_1$) is bordered by the four circles through the origin, orthogonally intersecting each other.  }\label{fig1}
\end{center}
\end{figure}

 For each $(k,\ell)\in\mathbb D_2$ 
 define $\phi_{k,\ell}\colon  \overline{U}\to\mathbb C$ by
\begin{equation}\label{formula}\phi_{k,\ell}(z)= \frac{1}{z+k+\sqrt{-1}\ell}.\end{equation} 
Note that $\phi_{k,\ell}|_{U}$ is univalent, $\phi_{k,\ell}(U)=U_{k,\ell}$,
$(\phi_{k,\ell}|_U)^{-1}=H|_{U_{k,\ell}}$  and $\phi_{k,\ell}(\overline{U})=\overline{\phi_{k,\ell}(U)}=\overline{U_{k,\ell}}\subset \overline U$.
We extend $\phi_{k,\ell}$ univalently to the outside of $\overline{U}$. For $r>0$ let
 \[U(r)=\left\{ x+\sqrt{-1}y\in \mathbb C\colon -\frac{1}{2}-r< x<\frac{1}{2}+r,\ -\frac{1}{2}-r< y< \frac{1}{2}+r\right\}.\]
 \begin{lem}\label{extend}If $0<r_0<1/2$, then for every $(k,\ell)\in\mathbb D_2$  the map
 
\[\widetilde\phi_{k,\ell}\colon z\in U(r_0)\mapsto \frac{1}{z+k+\sqrt{-1}\ell}\in\mathbb C\]
is well-defined, univalent and satisfies
$\max_{z\in\overline{U}}|D \widetilde\phi_{ k,\ell}(z)|<2/3$.
 If
 $0<r\leq r_0$ then $\widetilde\phi_{k,\ell}(U(r))\subset U(r)$.

 \end{lem}

\begin{figure}
\begin{center}
\includegraphics[height=10cm,width=11cm]{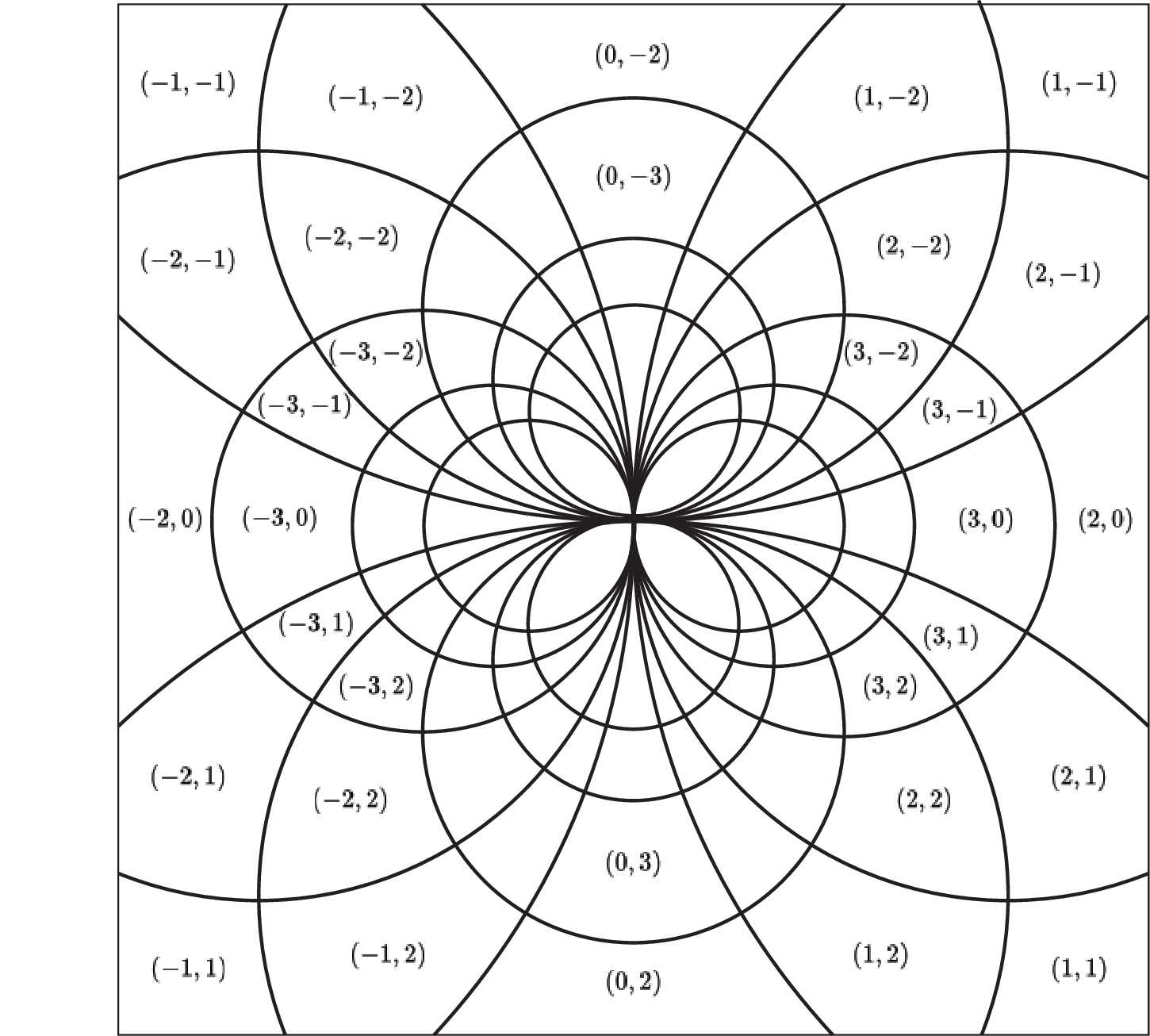}
\caption
{The collection $\{[c_1]\colon c_1=(k,\ell)\in\mathbb D_1 \}$ of $1$-cylinders tessellates $U$. }\label{fig2}
\end{center}
\end{figure}
 
 \begin{proof}
 Since $(k,\ell)\in\mathbb D_2$,
for any $z\in U(r_0)$ we have $z+k+\sqrt{-1}\ell\neq0.$ Hence   $\widetilde\phi_{ k,\ell}$ is well-defined and univalent. For any $z\in\overline{U}$ we have
\[|D \widetilde\phi_{ k,\ell}(z)|=\frac{1}{|z+k+\sqrt{-1}\ell|^2}\leq\frac{1}{(k-1/2)^2+(\ell-1/2)^2} < \frac{2}{3}.\]
We also have $|z+k+\sqrt{-1}\ell|\geq\sqrt{(|k|-1)^2+(|\ell|-1)^2}$ for all $z\in U(r_0)$.
If $0<r\leq r_0$, then for any $z\in U(r)$ there exists $z'\in\overline{U}$ with $|z-z'|<r$.
  Since $(|k|-1)^2+(|\ell|-1)^2\geq1$ we have
 \[\begin{split}|\widetilde\phi_{k,\ell}(z)-\phi_{k,\ell}(z')|&=\left|\frac{z'-z}{(z+k+\sqrt{-1}\ell)(z'+k+\sqrt{-1}\ell)}\right|\\
 &\leq\frac{|z'-z|}{(|k|-1)^2+(|\ell|-1)^2}<r.\end{split}\]
 Together with $\phi_{k,\ell}(z)\in\phi_{k,\ell}(\overline{U})\subset \overline{U}$ we obtain
 $\widetilde\phi_{k,\ell}(U(r))\subset U(r)$.\end{proof}
We fix $r_0\in(0,1/2)$ and set $\widetilde U=U(r_0/2)$. By Lemma~\ref{extend},
 for every $(k,\ell)\in \mathbb D_2$ the analytic extension of $\phi_{k,\ell}$ to $U(r_0)$ exists, which we denote by $\widetilde\phi_{k,\ell}$, and satisfies 
  \begin{equation}\label{inc-eq}\widetilde{\phi}_{k,\ell}(U(r_0) )\subset U(r_0)\ \text{ and }  \ \widetilde{\phi}_{k,\ell}(\widetilde U)\subset\widetilde U.\end{equation}


\begin{lem}\label{IFS-Con}The collection $\{\phi_{k,\ell}\colon(k,\ell)\in\mathbb D_2\}$ is a $2$-decaying conformal IFS on $\overline{U}$. \end{lem}
\begin{proof}
For any pair $(k_1,\ell_1)$, $(k_2,\ell_2)$ of elements of $\mathbb D_2$ we have  $\phi_{k_1,\ell_1}(U)\cap \phi_{k_2,\ell_2}(U)=U_{k_1,\ell_1}\cap U_{k_2,\ell_2}$. Hence (H1) implies (A1). Condition (A2) follows from the second inclusion in \eqref{inc-eq}.
By Lemma~\ref{extend}, for every $(k,\ell)\in\mathbb D_2$ we have
$\|D \widetilde\phi_{ k,\ell}\|_{\overline{U}}< 2/3<1$,
 and so (A3) holds. Hence $\{\phi_{k,\ell}\colon(k,\ell)\in\mathbb D_2\}$ is an infinite conformal IFS on $\overline{U}$. It is a $2$-decaying IFS 
 by virtue of the formula \eqref{formula} and (H5) that immediately yields (A4).\end{proof}


\begin{lem}\label{cylind-lem}For every $n\in\mathbb N$ and every $(c_1,\ldots,c_n)\in \mathbb D_2^n$, we have 
\[[c_1,\ldots, c_n]=\phi_{c_1\cdots c_n}(U).\]\end{lem}
\begin{proof}From (H4), if $n\geq2$ and $(c_1,\ldots,c_n)\in \mathbb D_2^n$ then
$[c_1,\ldots,c_n]=\phi_{c_1}([c_2,\ldots,c_n])$. A recursive use of this equality yields
$
[c_1,\ldots,c_n]=\phi_{c_1c_2}([c_3,\ldots,c_n])=\cdots=\phi_{c_1\cdots c_{n-1}}([c_n])$.
The last expression equals $\phi_{c_1\cdots c_n}(U)$ because
$[c_1]=\phi_{c_1}(U)$ holds
for every $c_1\in  \mathbb D_2$.
\end{proof}

\subsection{Proofs of Theorems~\ref{Main1} and \ref{Main2}}\label{mainpf}
Let $\Phi=\{\phi_{k,\ell}\colon(k,\ell)\in\mathbb D_2\}$ denote the 
$2$-decaying conformal IFS on $\overline{U}$ associated with the Hurwitz continued fraction as in $\S$\ref{basic-def}.
Let $S$ be an infinite subset of $\mathbb Z(\sqrt{-1})$.
Lemma~\ref{cylind-lem} implies
$F(S)\subset \bigcup_{n=0 }^\infty H^{-n}(F_\Phi(S))$.
Since the Hurwitz map $H$ in \eqref{H-map} is piecewise Lipschitz and Hausdorff dimensions do not change under the action of bi-Lipschitz homeomorphisms, we have
 \[\dim_{\rm H}F(S)\leq\dim_{\rm H} F_\Phi(S).\]
 
 Let $f\colon\mathbb{N}\rightarrow [\min |S|, \infty)$ satisfy $\lim_{n\to \infty}f(n)=\infty$. 
Lemma~\ref{cylind-lem} also implies that $F(S,f)$ contains $F_\Phi(S,f)$, and so
 \[\dim_{\rm H}F(S,f)\geq\dim_{\rm H}F_\Phi(S,f).\] Combining these two inequalities and then using Theorem~\ref{Main100},
 we obtain
 $\dim_{\rm H}F(S)=\dim_{\rm H}F(S,f)=\tau(|S|)/2$ as required in Theorems~\ref{Main1} and \ref{Main2}. \qed

\subsection{Proof of Theorem~\ref{Main-cor}}\label{pfthmcor}
By Theorem~\ref{Main1}, the set
\[\left\{z\in U\setminus\mathbb Q(\sqrt{-1})\colon \lim_{n\to\infty}|c_n(z)|=\infty \right\}\]
is of Hausdorff dimension $\tau(|\mathbb Z(\sqrt{-1})|)/2$.
In order to show 
$\tau(|\mathbb Z(\sqrt{-1})|)=2$,
we use the following formula for convergence exponent
(see \cite[Section~2]{PS72}):
if $S=\{x_n\colon n\in\mathbb N\}$ is a set of positive reals with $|x_n|\leq |x_{n+1}|$ for every $n\in\mathbb N$ and $x_n\to\infty$ as $n\to\infty$, then
\begin{equation}\label{formula-tau}\tau(|S|)=\limsup_{n\to \infty}\frac{\log n}{\log x_n}.\end{equation}

Write $\mathbb Z(\sqrt{-1})=\{z_n\colon n\in\mathbb N\}$, $|z_1|\leq |z_2|\leq \cdots$.  
Let $n,N\in\mathbb N$ satisfy $n\geq10$ and \begin{equation}\label{count-eq2}
(2N+1)^2< n \leq (2N+2)^2.\end{equation} 
Since the number of Gaussian integers contained in the square $[-N,N]^2\subset\mathbb R^2=\mathbb C$ is $(2N+1)^2$,
the first inequality in \eqref{count-eq2} implies $z_n\notin[-N,N]^2$, and so
\begin{equation}\label{count-eq3}N<|z_n|.\end{equation}
Since the number of Gaussian integers contained in  $[-N-1,N+1]^2$ is $(2N+3)^2$,
  The second inequality in \eqref{count-eq2} implies $z_n\in[-N-1,N+1]^2$, and so
 \begin{equation}\label{count-eq4}|z_n|\leq \sqrt{2}(N+1).\end{equation}
From  \eqref{count-eq3} and \eqref{count-eq4} it follows that for any $\varepsilon>0$ there exists $n(\varepsilon)\geq1$ such that for every $n\geq n(\varepsilon)$ we have
  $n^{\frac{1}{2}-\varepsilon}\leq |z_n|\leq n^{\frac{1}{2}+\varepsilon}.$
From this estimate and \eqref{formula-tau} we obtain
\[\tau(|\mathbb Z(\sqrt{-1})|)=\limsup_{n\to\infty}\frac{\log n}{\log |z_{n}|}=2,\]
as required.
\qed

\subsection*{Acknowledgments}
We thank Gerardo Gonz\'alez Robert for fruitful discussions.
This research was supported by the JSPS KAKENHI 23K20220 and 25K17282.


\begin{thebibliography}{99}

\bibitem{Bos87} Wieb Bosma, Optimal continued fractions, Indag. Math. 
{\bf 49} (1987), no. 4, 353--379.  





\bibitem{Bow79} Rufus Bowen, {\it Hausdorff dimension of quasicircles}, 
    Inst. Hautes \'Etudes Sci. Publ. Math. {\bf  50} (1979), 11-25.


 
\bibitem{BGH25}Yann Bugeaud, Gerardo Gonz\'alez Robert, and Mumtaz Hussain, {\it Metrical properties of Hurwitz continued fractions},  
 Adv. Math. {\bf 468} (2025), Paper No. 110208.


\bibitem{CWW13} Chun-Yun Cao,
Bao-Wei Wang, and Jun Wu,
{\it The growth speed of digits in infinite iterated function systems},
Studia Math. {\bf 217} (2013), no. 2, 139--158.

\bibitem{CG93}
Lennart Carleson and Theodore W. Gamelin, 
{\it Complex Dynamics.} (English summary)
Universitext: Tracts in Mathematics. Springer-Verlag, New York, 1993. 

\bibitem{Cus90} Thomas W. Cusick, {\it Hausdorff dimension of sets of continued fractions}, 
  Quart. J. Math. Oxford {\bf 41} (1990), 277--286.

\bibitem{DK99} Karma Dajani and 
 Cor Kraaikamp, {\it The mother of all continued fractions}, Colloq. Math.
 {\bf 84/85} (2000), part 1, 109--123. 

\bibitem{DN} S. G. Dani and Arnaldo Nogueira, 
{\it Continued fractions for complex numbers and values of binary quadratic forms}, Trans. Amer. Math. Soc. {\bf 366} (2014), no. 7, 3553--3583.





\bibitem{Fal14} Kenneth Falconer, {\it Fractal geometry. Mathematical foundations and applications}, Third edition. John Wiley $\&$ Sons, Ltd., Chichester, 2014.
 
\bibitem{Ger}Gerardo Gonz\'alez Robert, {\it Good's theorem for Hurwitz continued fractions}, 
Int. J. Number Theory {\bf 16} (2020), no. 7,
1433--1447. 
\bibitem{GHST}
 Gerardo Gonz\'alez Robert, Mumtaz Hussain, Nikita Shulga, and Hiroki Takahasi, {\it Infinite iterated function systems with restricted slowly growing digits}, J. Math. Anal. Appl. {\bf 
549} (2025), no. 1, 129478

\bibitem{Good1941} I. J. Good, {\it The fractional dimensional theory of continued fractions},  Proc. Cambridge Philos. Soc. {\bf 37} (1941), 199--228.

\bibitem{Hir73}  K. E. Hirst,
{\it Continued fractions with sequences of partial quotients},  Proc. Amer. Math. Soc.
{\bf 38} (1973), 221--227.



\bibitem{HuWu09} X. H. Hu and Jun Wu, {\it Continued fractions with sequences of partial quotients over the field of Laurent series},  Acta Arith. {\bf 136} (2009), no. 3, 201--211. 

\bibitem{H87} Adolf Hurwitz, {\it \"Uber die Entwicklung complexer Gr\"ossen in Kettenbr\"uche}, 
Acta Math. {\bf 11} (1887), no. 1--4, 187--200. 

\bibitem{H1902} Julius Hurwitz, 
{\it \"Uber die Reduction der Bin\"aren Quadratischen
Formen mit Complexen Coefficienten und Variabeln}, 
Acta Math. {\bf 25} (1902), no. 1, 231--290. 


\bibitem{Jarnik1}  Vojt\u{e}ch Jarn\'ik, {\it Zur metrischen Theorie der diophantischen Approximationen}, Prace Mat.-Fiz. {\bf 36} (1928-29), 81--106.





\bibitem{KST76} Ryuji Kaneiwa, Iekata Shiokawa, and Jun-ichi Tamura, {\it Some properties of complex continued fractions}, Comment. Math. Univ. St. Paul, {\bf 25} (1976/77), no. 2, 129--143.


\bibitem{Kra91}  Cor Kraaikamp, {\it A new class of continued fraction
expansions}, Acta Arith. {\bf 57} (1991), no. 1, 1--39. 
\bibitem{Lak73} Richard B. Lakein,
{\it Approximation properties of some complex continued fractions},
Monatshefte f\"ur Mathematik, {\bf 77}, 396--403 (1973)

\bibitem{Lak75} Richard B. Lakein,
{\it A continued fraction proof of Ford's theorem on complex rational approximations}, J. f\"ur die reine und angew. Math. {\bf 272} (1975) 1--13.

\bibitem{LeV52} W. J. LeVeque,
{\it Continued fractions and approximations in $k(i)$. I, II},
Nederl. Akad. Wetensch. Proc. Ser. A. {\bf 55} = Indag. math. {\bf 14} (1952), 526--535, 536--545.

\bibitem{MU96}R. Daniel Mauldin and Mariusz Urba\'nski, {\it Dimensions and measures in infinite iterated function systems}, Proc. London Math. Soc. (3) {\bf 73} (1996), no. 1, 105-154.


\bibitem{MauUrb}R. Daniel Mauldin and Mariusz Urba\'nski, 
 {\it Graph directed Markov systems. Geometry and dynamics of limit sets.} Cambridge Tracts in Mathematics, {\bf 148} Cambridge University Press, Cambridge, 2003. 

\bibitem{N24} Yuto Nakajima, {\it Transversal family of non-autonomous conformal iterated function systems}, J. Fractal Geom. {\bf 11} (2024), no. 1--2, 57--84.

\bibitem{NT} Yuto Nakajima and Hiroki Takahasi, {\it Hausdorff dimension of sets with restricted, slowly growing partial quotients in semi-regular continued fractions}, J. Math. Soc. Japan, https://doi.org/10.2969/jmsj/93059305

\bibitem{Per50} Oskar Perron, {\it Die Lehre von den Kettenbr\"uchen,}
Second edition. Chelsea Publishing Co., New York 1950.

\bibitem{PS72} George P\'olya and Gobor Szeg\H{o}, {\it Problems and Theorems in Analysis} Vol I. Berlin: Springer-Verlag, 1972.
\bibitem{RU}  Lasse Rempe-Gillen and  
Mariusz Urba\'nski, {\it Non-autonomous conformal iterated function systems and Moran-set constructions}, Trans. Amer. Math. Soc. {\bf 368} (2016), no. 3, 1979--2017.

\bibitem{S75} Asmus L. Schmidt,
{\it Diophantine approximation of complex numbers}, Acta Math. {\bf 134} (1975), 1--85. 
\bibitem{T23} Hiroki Takahasi, 
{\it Hausdorff dimension of sets with restricted, slowly growing partial quotients}, Proc. Amer. Math. Soc. {\bf 151} (2023), no. 9, 3645--3653.


\bibitem{TT16} Liang Tang and Ting Zhong, {\it Some dimension relations of the Hirst sets in regular and generalized continued fractions},  J. Number Th. {\bf 167} (2016), 128--140. 

\bibitem{WW08} Bao-Wei Wang and Jun Wu,  {\it A problem of Hirst on continued fractions with sequences of partial quotients}, Bull. Lond. Math. Soc. {\bf 40} (2008), no. 1, 18--22.

\bibitem{WW08b} Bao-Wei Wang and Jun Wu, 
{\it Hausdorff dimension of certain sets arising in continued fraction expansions}, Adv. Math. 
{\bf 218} (2008), no. 5, 1319--1339.

\end{thebibliography}
\end{document}